\documentclass[11pt,oneside]{amsart}
\usepackage{epsfig}
\usepackage{amsmath,ifthen, amsfonts, amssymb, srcltx, amsopn}
\usepackage{graphicx}
\usepackage[small]{caption}
\usepackage{enumerate}
\usepackage{txfonts}

\newcommand{\showcomments}{yes}
\renewcommand{\showcomments}{no}

\newsavebox{\commentbox}
\newenvironment{com}%
{\ifthenelse{\equal{\showcomments}{yes}}%
{\footnotemark
    \begin{lrbox}{\commentbox}
    \begin{minipage}[t]{1.25in}\raggedright\sffamily\tiny
    \footnotemark[\arabic{footnote}]}
{\begin{lrbox}{\commentbox}}}%
{\ifthenelse{\equal{\showcomments}{yes}}%
{\end{minipage}\end{lrbox}\marginpar{\usebox{\commentbox}}}
{\end{lrbox}}}

\newtheorem{thm}{Theorem}[section]
\newtheorem{lem}[thm]{Lemma}

\newtheorem{cor}[thm]{Corollary}
\newtheorem{conj}[thm]{Conjecture}

\theoremstyle{definition}
\newtheorem{defn}[thm]{Definition}

\newtheorem{rem}[thm]{Remark}
\newtheorem{exmp}[thm]{Example}

\newtheorem{tecrem}[thm]{Technical Remark}

\newcommand{\cev}[1]{\reflectbox{\ensuremath{\vec{\reflectbox{\ensuremath{#1}}}}}}

\newcommand{\field}[1]{\mathbb{#1}}%
\newcommand{\integers}{\ensuremath{\field{Z}}}

\newtheorem*{MainRH}{Theorem A}
\newtheorem*{CorRHVertexParabolics}{Corollary B}
\newtheorem*{MainQC}{Main Theorem C}
\newtheorem*{LQCap}{Theorem D}
\newtheorem*{CorManifold}{Corollary E}
\begin{document}

\title[Quasiconvexity and Relatively Hyperbolic Groups that Split]
{Quasiconvexity and Relatively Hyperbolic Groups that Split}

\author[H.~Bigdely]{Hadi Bigdely}
      \address{Dept. of Math \& Stats.\\
               McGill University \\
               Montreal, Quebec, Canada H3A 2K6 }
\email{bigdely@math.mcgill.ca}
\author[D. T.~Wise]{Daniel T. Wise}
      \email{wise@math.mcgill.ca}
\thanks{Research supported by NSERC}


\date{\today}

\begin{com}
{\bf \normalsize COMMENTS\\}
ARE\\
SHOWING!\\
\end{com}

\begin{abstract}
We explore the combination theorem for a group $G$ splitting as a graph of relatively hyperbolic groups. Using the fine graph approach to relative hyperbolicity, we find short proofs of the relative hyperbolicity of $G$ under certain conditions. We then provide a criterion for the relative quasiconvexity of a subgroup $H$ depending on the relative quasiconvexity of the intersection of $H$ with the vertex groups of $G$. We give an application towards local relative quasiconvexity.
\end{abstract}

\maketitle

\tiny
\normalsize
The goal of this paper is to examine  relative hyperbolicity  and quasiconvexity
in graphs of  relatively hyperbolic vertex groups with almost malnormal quasiconvex edge groups.
The paper hinges upon the observation that if $G$ splits as a graph of relatively hyperbolic groups
with malnormal relatively quasiconvex edge groups, then a fine hyperbolic graph for  $G$ can be built
from fine hyperbolic graphs for the vertex groups. This leads to short proofs of the relative hyperbolicity of $G$
as well as to a concise criterion for the relative quasiconvexity of a subgroup $H$ of $G$.

Bestvina and Feighn proved a combination theorem that characterized the hyperbolicity of groups splitting as  graphs of  hyperbolic groups \cite{BF92}. Their geometric characterization is akin to the flat plane theorem characterization of hyperbolicity for actions on CAT(0) spaces, and leads to explicit positive results, especially in an ``acylindrical'' scenario where some form of malnormality is imposed on the edge groups. The Bestvina-Feighn combination theorem has been revisited multiple times in a hyperbolic setting,
and more recently in a relatively hyperbolic context but through diverse methods.

Dahmani  proved a combination theorem for relatively hyperbolic groups using the convergence group approach
\cite{Dahmani03}. Later Alibegovi\'{c}  proved similar results in \cite{Alibegovic05}
using a method generalizing parts of the Bestvina-Feighn approach.  Osin reproved Dahman's result in the general context of  relative Dehn functions  \cite{Osin2006}.
Most recently, Mj and Reeves gave a generalization
of the Bestvina-Feighn combination theorem that follows Farb's approach but uses a generalized ``partial electrocution'' \cite{MjReeves2008}. Their result appears to be a far-reaching generalization at the expense of complex geometric language.

Our own results revisit these relatively hyperbolic generalizations, and we offer a very concrete approach
employing Bowditch's fine hyperbolic graphs.
The most natural formulation of our main combination theorem (proven as Theorem~\ref{first}) is as follows:
\begin{MainRH}[Combining Relatively Hyperbolic Groups Along Parabolics]\label{thm:MainRH}
Let $G$ split as a finite graph of groups. Suppose each vertex group is relatively hyperbolic and each edge group is parabolic in its vertex groups. Then $G$ is hyperbolic relative to $\mathbb Q=\{Q_1,\dots, Q_j\}$ where each $Q_i$ is the stabilizer of a ``parabolic tree''. (See Definition~\ref{parabolic tree}.)
\end{MainRH}

A simplistic example illustrating Theorem~A is an amalgamated product $G=G_1*_C G_2$ where each $G_i = \pi_1M_i$ and $M_i$ is a cusped hyperbolic manifold with a single boundary torus $T_i$. And $C$ is an arbitrary common subgroup of $\pi_1T_1$ and $\pi_1T_2$. Then $G$ is hyperbolic relative to $\pi_1 T_1*_C \pi_1 T_2$.

We note that Theorem~A is more general than results in the same spirit that were obtained by Dahmani, Alibegovi\'{c}, and Osin.
In particular, they require that edge groups be maximal parabolic on at least one side, but we do not.
We believe that Theorem~A could be deduced from the results of Mj-Reeves.

In Section~\ref{sec:improvements using yang}, we employ work of Yang \cite{Yang11} on extended peripheral structures,
to obtain the following seemingly more natural corollary of Theorem~A which is proven as Corollary~\ref{hyp Qc combination}:

\begin{CorRHVertexParabolics}\label{cor:CorRHVertexParabolics}
Let $G$ split as a finite graph of groups. Suppose
\begin{itemize}
\item[(a)] Each $G_{\nu}$ is hyperbolic relative to $\mathbb P_{\nu}$;
\item[(b)] Each $G_e$ is total and relatively quasiconvex in $G_{\nu}$;
\item[(c)] $\{G_e: e~ \textrm{is attached to } \nu\}$ is almost malnormal in $G_{\nu}$ for each vertex $\nu$.
\end{itemize}
Then $G$ is hyperbolic relative to $\bigcup_{\nu} \mathbb P_{\nu}-\{\text{repeats}\}$.
\end{CorRHVertexParabolics}
\noindent The ``omitted repeats'' in the conclusion  of Corollary~B
refer to (some of) the parabolic subgroups of vertex groups that are identified through an edge group.

It is not clear whether Corollary~B could be obtained using the method of Dahmani, Alibegovi\'{c}, or Osin.
However, we suspect it could be extracted from the result of Mj-Reeves.

\begin{defn}(Tamely generated)
Let $G$ split as a graph of groups with relatively hyperbolic vertex groups.
A subgroup $H$ is \emph{tamely generated} if
the induced graph of groups $\Gamma_H$ \begin{com} unfortunately we used $\Gamma_H$ notation at the end...\end{com}
has a $\pi_1$-isomorphic subgraph of groups  $\Gamma'_H$
that is a finite graph of groups each of whose vertex groups is
 relatively quasiconvex in the corresponding vertex group of $G$.

Note that $H$ is tamely generated when $H$ is finitely generated
and there are finitely many $H$-orbits of vertices $v$ in $T$ with $H_v$ nontrivial,
and each such $H_v$ is relatively quasiconvex in $G_v$.
However the above condition is not necessary.
For instance, let $G=F_2\times \integers_2$, and consider a splitting where $\Gamma$ is a bouquet of two circles,
and each vertex and edge group is isomorphic to $\integers_2$. Then  every f.g. subgroup $H$ of $F_2 \times \integers_2$ is tamely generated,
but no subgroup containing $\integers_2$ satisfies the condition that there are finitely many $H$-orbits of vertices $\omega$ with $H_\omega$ nontrivial.
\end{defn}
The geometric construction proving Theorem~A allows us to give a simple criterion for quasiconvexity of a subgroup $H$
relative to $\mathbb Q$. Again, coupling this with Yang's work, we obtain (as Theorem~\ref{thm:strongest result})
 the following criterion for quasiconvexity relative to $\mathbb P$:

\begin{MainQC}[Quasiconvexity Criterion]
Let $G$ be hyperbolic relative to $\mathbb P$ where each $P\in \mathbb P$ is finitely generated. Suppose $G$ splits as a finite graph of groups.
 Suppose
\begin{enumerate}[(a)]
\item Each $G_e$ is total in $G$;
\item Each $G_e$ is relatively quasiconvex in $G$;
\item Each $G_e$ is almost malnormal in $G$.
\end{enumerate}
Let $H \leq G$ be tamely generated. Then $H$ is relatively quasiconvex in $G$.
\end{MainQC}

Recall that $G$ is \emph{locally relatively quasiconvex} if each finitely generated subgroup $H$ of $G$ is quasiconvex relative to
the peripheral structure of $G$.
I.~Kapovich first recognized that hyperbolic limit groups are locally relatively quasiconvex \cite{Kapovich2002},
 and subsequently  Dahmani proved that all limit groups are locally relatively quasiconvex \cite{Dahmani03}.

 A group $P$ is \emph{small} if there is no embedding $F_2\hookrightarrow P$,
 and $G$ has a \emph{small hierarchy} if it can be built from small subgroups by a sequence of AFP's and HNN's
 along small subgroups (see Definition~\ref{defn:small}).
When $\mathbb P$ is a collection of free-abelian groups,
 the following inductive consequence of Corollary~\ref{cor:lqc noetherian combination}
 generalizes Dahmani's result.
\begin{LQCap}
Let $G$ be hyperbolic relative to a collection of Noetherian subgroups $\mathbb P$ and suppose $G$ has a
small hierarchy. Then $G$ is locally relatively quasiconvex.
\end{LQCap}

Although Theorem~D is implicit in Dahmani's work, we  believe Theorem~C is new.

\begin{figure}\centering
\includegraphics[width=.7\textwidth]{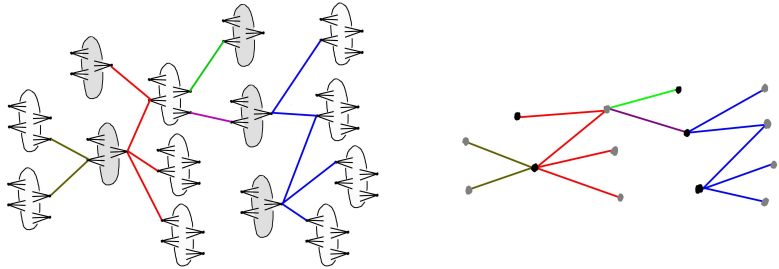}
\caption{A fine graph $K_G$ for $G=A*_C B$ is built from copies of fine graphs $K_A$ and $K_B$ for $A$ and $B$ by gluing new edges together along vertices stabilized by $C$. The parabolic trees of $T$ are images of  trees formed from the new edges in $K_G$. We obtain a fine hyperbolic graph $\bar K_G$ with finite edge stabilizers as a quotient $K_G\rightarrow \bar K_G$.
 \label{fig:NewFineGraph}}
\end{figure}

{\bf The main construction and its application:}
Although we  work in somewhat greater generality,
let us focus on the simple case of an amalgamated product $G=A*_C B$ where $A,B$ are relatively hyperbolic and $C$ is parabolic
on each side.
The central theme of this paper is a construction that builds a fine hyperbolic graph $\bar K_G$ for $G$
from fine hyperbolic graphs $K_A$ and $K_B$ for $A,B$. (See Figure~\ref{fig:NewFineGraph}.)
This is done in two steps: Guided by the Bass-Serre tree, we first construct a graph $K_G$ which is a tree of spaces whose vertex spaces are copies of $K_A$ and $K_B$, and whose edge spaces are ordinary edges. Though $K_G$ is fine and hyperbolic, its edges have infinite stabilizers. We remedy this by quotienting these edge spaces
to form the fine hyperbolic graph $\bar K_G$. The vertices of $\bar K_G$ are quotients of ``parabolic trees'' in $K_G$.
The fine hyperbolic graph $\bar K_G$ quickly proves that $G$ is hyperbolic relative to the collection $\mathbb Q$ of subgroups stabilizing  parabolic trees. Variations on the construction, hypotheses on the edge groups, and interplay with previous work on peripheral structures, leads to a variety of relatively hyperbolic conclusions. The simplest and most immediate in the case above,
is that $G$ is hyperbolic relative to $\mathbb P_G = \mathbb P_A \cup \mathbb P_B -\{C\}$
 when $C$ is maximal parabolic on each side and $A,B$ are hyperbolic relative to $\mathbb P_A$ and $\mathbb P_B$.

Our primary application is to give an easy criterion to recognize quasiconvexity.
A subgroup $H$ is relatively quasiconvex in $G$ if there is an $H$-cocompact quasiconvex subgraph $\bar L\subset \bar K_G$ of the fine hyperbolic $G$-graph.
The tree-like nature of our graph $\bar K_G$, permits us to naturally build the quasiconvex $H$-graph $\bar L$.
When $H$ is relatively quasiconvex, there are finitely many $H$-orbits of nontrivially $H$-stabilized vertices in the Bass-Serre tree $T$, and each of these stabilizers is relatively quasiconvex in its vertex group.
Choosing finitely many quasiconvex subgraphs in the corresponding copies of $K_A$ and $K_B$,
we are able to combine these together to form $L$ in $K_G$ and then to form a quasiconvex $H$-subgraph
$\bar L$ in $\bar K_G$.

We conclude by mentioning the following consequence of Corollary~\ref{cor:first}
that is a natural consequence of the viewpoint developed in this paper.
\begin{CorManifold}
Let $M$ be a compact irreducible 3-manifold.
And let $M_1,\ldots, M_r$ denote the graph manifolds obtained by removing each (open) hyperbolic piece
in the geometric decomposition of $M$.
Then  $\pi_1M$ is hyperbolic relative to $\{\pi_1M_1,\ldots, \pi_1M_r\}$.
\end{CorManifold}
As explained to us by the referee, the relative hyperbolicity of
$\pi_1(M)$ was previously proved by Drutu-Sapir using work of
Kapovich-Leeb.  This previous proof is deep as it uses the structure of the
asymptotic cone due to Kapovich-Leeb together with the technical
proof of Drutu-Sapir that asymptotically tree graded groups are
relatively hyperbolic \cite{KapovichLeeb1995,DrutuSapir2005}.


\section{Combining Relatively Hyperbolic Groups along Parabolics}\label{sec:rhg}
The class of relatively hyperbolic groups was introduced
by Gromov \cite{Gromov87} as a generalization of the class of fundamental groups
of complete finite-volume manifolds of pinched negative sectional curvature.
Various approaches to relative hyperbolicity were developed
by Farb \cite{Farb98}, Bowditch \cite{Bowditch99RH} and Osin \cite{OsinBook06}, and as surveyed by Hruska \cite{HruskaRelQC}, these notions are equivalent for finitely generated groups.
\begin{com} What about Bumagin? IS she more appropriate or should be added?\end{com}
 We follow Bowditch's approach:

\begin{defn}[Relatively Hyperbolic]A {\em{circuit}}
in a graph is an embedded cycle. A graph $\Gamma$ is {\em{fine}} if each edge of $\Gamma$ lies in finitely many circuits of length $n$ for each $n$.

A group G is \emph{hyperbolic
relative to a finite collection of subgroups} $\mathbb{P}$ if $G$ acts cocompactly(without inversions)\begin{com} take these two out and say cofinitely on the set of edges\end{com}
on a connected, fine, hyperbolic graph $\Gamma$ with finite edge stabilizers, such that each element of $\mathbb P$ equals the stabilizer of a vertex of $\Gamma$, and moreover, each infinite vertex stabilizer is conjugate to a unique element of $\mathbb P$.
We refer to a connected, fine, hyperbolic graph $\Gamma$ equipped with
such an action as a \emph{$(G; \mathbb{P})$-graph}. Subgroups of G that are conjugate into
subgroups in $\mathbb{P}$ are \emph{parabolic}.\begin{com}Bowditch def requires fg for parabolics\end{com}
\end{defn}
\begin{tecrem}\label{tr}
Given a finite collection of parabolic subgroups $\{A_1,\dots,A_r\}$\begin{com}Seems not clear\end{com}, we choose $\mathbb P$ so that there is a prescribed choice of parabolic subgroup $P_i\in \mathbb P$ so that $A_i$ is \emph{``declared''} to be conjugate into $P_i$. This is automatic for an infinite parabolic subgroup $A$ but for finite subgroups there could be ambiguity. One way to resolve this is to revise the choice of $\mathbb P$ as follows:
For any finite collection of parabolic subgroups $\{A_1,\dots,A_r\}$ in $G$, we moreover assume each $A_i$ is conjugate to a subgroup of $\mathbb P$  and we  assume that no two (finite) subgroups in $\mathbb P$ are conjugate.
We note that finite subgroups can be freely added to or omitted from the peripheral structure of $G$
(see e.g. \cite{MartinezPedrozaWise2011}).
\end{tecrem}

\begin{com}
A more general, and more sensible route to follow is:
Let $G$ split as a finite graph of groups, where each $G_v$ is relatively hyperbolic and is assigned a fine hyperbolic graph $K_v$,
and where each finite subgroup $P$ of $\mathbb P_v$ is assigned a ``chosen'' vertex $a_P$ of $K_v$ (infinite $P$ has $a_P$ uniquely determined already), and moreover each finite edge group (representative) of $G_v$  is assigned a ``chosen'' vertex.

The vertices of $F$ are the translates of $a_P$,
and an edge of $F$ are attached to its chosen vertices in $K_{\iota(e)}$ and $K_{\tau(e)}$.

This makes the embedding $F\subset K$ more natural.
\end{com}
\begin{com}$Gamma$ is used for graph of groups and and K for (G;P)-graph\end{com}
\begin{defn}[Parabolic tree]\label{parabolic tree}
Let $G$ split as a finite graph
of groups where each vertex group $G_\nu$ is hyperbolic relative to $\mathbb P_\nu$, and where each edge group $G_e$
embeds as a parabolic subgroup of its two vertex groups. Let $T$ be the Bass-Serre tree.
Define the \emph{parabolic forest} $F$ by:
\begin{enumerate}
\item A \emph{vertex} in $F$ is a pair $(u,P)$ where $u\in T^0$ and $P$ is a $G_u$-conjugate of an element of $\mathbb P_u$.

\item An \emph{edge} in $F$ is a pair $(e,G_e)$ where $e$ is an edge of $T$ and $G_e$ is its stabilizer.

\item The edge $(e,G_e)$ is \emph{attached} to $(\iota(e), \iota(P_e))$ and $(\tau(e), \tau(P_e))$ where $\iota(e)$ and $\tau(e)$ are the initial and terminal vertex of $e$ and $\iota(P_e)$ is the $G_{\iota(e)}$-conjugate of an element of $\mathbb P$ that is declared to contain $G_e$. Likewise for $(\tau(e), \tau(P_e))$. We arranged for this unique determination in Technical Remark~\ref{tr}.

\end{enumerate}

Each component of $F$ is a \emph{parabolic tree} and the map $F\rightarrow T$ is injective on the set of edges, and in particular each parabolic tree embeds in $T$.
Let $S_1,\dots,S_j$ be representatives of the finitely many orbits of parabolic trees under the $G$ action on $F$. Let $Q_i=stab(S_i)$, for each $i$.
\end{defn}
\begin{thm}[Combining Relatively Hyperbolic Groups Along Parabolics]\label{first}
Let $G$ split as a finite graph $\Gamma$ of groups. Suppose each vertex group is relatively hyperbolic and each edge group is parabolic in its vertex groups. Then $G$ is hyperbolic relative to $\mathbb Q=\{Q_1,\dots, Q_j\}$.
\end{thm}
\begin{proof}
For $u\in \Gamma^0$, let $G_u$ be hyperbolic relative to $\mathbb P_u$ and let $K_u$ be a $(G_u; \mathbb{P}_u)$-graph. For each $P\in \mathbb P_u$, following the Technical Remark~\ref{tr}, we choose a specific vertex of $K_u$ whose stabilizer equals $P$. Note that, in general there could be more than one possible choice when $|P|<\infty$, but by Technical Remark~\ref{tr} we have a unique choice. Translating determines a ``choice'' of vertex for conjugates.

We  now construct a $(G; \mathbb Q)$-graph $\bar K$.
Let $K$ be the tree of spaces whose underlying tree is the Bass-Serre tree $T$ with the following properties:

\begin{enumerate}
\item Vertex spaces of $K$ are copies of appropriate elements in $\{K_u: u\in \Gamma^0\}$. Specifically, $K_\nu$ is a copy of $K_u$ where $u$ is the image of $\nu$ under $T\rightarrow \Gamma$.
\item Each edge space $K_e$ is an ordinary edge, denoted as an ordered pair $(e,G_e)$ that is attached to the vertices in  $K_{\iota(e)}$ and $K_{\tau(e)}$ that were chosen to contain $G_e$.

\end{enumerate}
Note that each $G_{\nu}$ acts on $K_\nu$ and there is a $G$-equivariant map $K\rightarrow T$. Let $\bar K$ be the quotient of $K$ obtained by contracting each edge space. Observe that $G$ acts on $\bar K$ and there is a $G$-equivariant map $K\rightarrow \bar K$. Moreover the preimage of each open edge of $\bar K$ is a single open edge of $K$.

We now show that $\bar K$ is a $(G; \mathbb Q)$-graph. Since any embedded cycle lies in some vertex space, the graph $\bar K$ is fine and hyperbolic. There are finitely many orbits of vertices in $K$ and therefore finitely many orbits of vertices in $\bar K$. Likewise, there are finitely many orbits of edges in $\bar K$. The stabilizer of an (open) edge of $\bar K$ equals the stabilizer of the corresponding (open) edge in $K$, and is thus finite. By construction, there is a $G$-equivariant embedding $F\hookrightarrow K$ where $F$ is the parabolic forest associated to $G$ and $T$. Finally, the preimage in $K$ of a vertex of $\bar K$ is precisely a parabolic tree and thus the stabilizer of a vertex of $\bar K$ is a conjugate of some $Q_j$.
\end{proof}

We now examine some conclusions that arise when the parabolic trees are small. An extreme case arises when the edge groups are isolated from each other as follows:

\begin{cor}\label{cor:first}
Let $G$ split as a finite directed graph of groups where each vertex group $G_{\nu}$ is hyperbolic relative to $\mathbb P_{\nu}$.
Suppose that:
\begin{enumerate}
\item Each edge group is parabolic in its vertex groups.
\item Each outgoing infinite edge group $G_{\vec{e}}$ is maximal parabolic in its initial vertex group $G_{\nu}$ and for each other incoming and outgoing infinite edge group $G_{{}\cev e }$ or $G_{\vec{d}}$ or $G_{\cev d}$, none of its conjugates lie in
    $G_{\vec{e}}$.
\end{enumerate}

Then $G$ is hyperbolic relative to $\mathbb P=\bigcup_{\nu} \mathbb P_{\nu}-\{\text{outgoing edge groups}\}$.\begin{com}There is a more complex version. Number the edges and assume no $G_{\vec{d}}$ or $G_{\cev {d}}$ with $d\leq e$ is conjugate into $G_{\hspace{.1mm}\cev{e}}$...
\end{com}
\end{cor}

\begin{proof}
We can arrange
 \begin{com} We must either use the more systematic definition of $F$ (hidden above),
 or alternately, we can observe that the parabolic trees have stabilizers that naturally split as graphs of (desired) parabolic
 subgroups over finite edge groups, and so we can decompose to get the claimed result...
 \end{com}
 for finitely stabilized edges of $F$ to be attached to distinct chosen vertices when they correspond to distinct edges of $T$. Thus, parabolic trees are singletons and/or $i$-pods consisting of edges that all terminate at the same vertex $\{(\nu,P^g)\}$ where $P\in \mathbb P_\nu$ and $g\in G_\nu$. Recall that an \emph{$i$-pod} is a tree consisting of $i$~edges glued to a central vertex.
\end{proof}

\begin{cor}\label{hypspli}
Let $G$ split as a finite graph of groups. Suppose each vertex group $G_{\nu}$ is hyperbolic relative to $\mathbb P_{\nu}$. For each $G_\nu$ assume that the collection $\{G_e: e \textrm{ is attached to } \nu \}$ is a collection of maximal parabolic subgroups of $G_\nu$. Then $G$ is hyperbolic relative to $\mathbb P=\bigcup_{\nu} \mathbb P_{\nu}-\{\text{repeats}\}$.
 Specifically, we remove an element of $\bigcup_{\nu} \mathbb P_{\nu}$ if it is conjugate to another one.\begin{com} Mark offers to put a def of repeats as a equivalent relation\end{com}
\end{cor}
The first two of the following cases were treated by Dahmani, Alibegovi\'{c}, and~Osin
\cite{Alibegovic05,Dahmani03,Osin2006}:
\begin{cor}\label{mlemma}
\begin{enumerate}
\item Let $G_1$ and $G_2$ be hyperbolic relative to $\mathbb P_1$ and $\mathbb P_2$. Let $G=G_1 \ast_{P_1={P_2'}} G_2$ where each $P_i\in \mathbb P_i$ and $P_1$ is identified with the subgroup ${P_2}'$ of $P_2$. Then $G$ is hyperbolic relative $\mathbb P_1\cup \mathbb P_2-\{P_1\}$.
\item  Let $G_1$ be hyperbolic relative to $\mathbb P$. Let $P_1\in \mathbb P$ be isomorphic to a subgroup ${P_2}'$ of a maximal parabolic subgroup $P_2$
not conjugate to $P_1$. Let $G = G_1*_{{P_1}^t={P_2}'}$ where ${P_1}^t=t^{-1}P_1t$. Then $G$ is
hyperbolic relative to $\mathbb P-\{P_1\}$.
\item Let $G_1$ be hyperbolic relative to $\mathbb P$. Let $P\in \mathbb P$ be isomorphic to ${P}'\leq P$. Let $G = G_1*_{{P}^t={P}'}$. Then $G$ is
hyperbolic relative to $\mathbb P\cup \langle P, t \rangle-\{P\}$.
\end{enumerate}
\end{cor}
\begin{rem}
Note that in this Corollary and some similar results when we say $P_i\in \mathbb P_i$, we mean if $P^g_i \in \mathbb P_i$ then replace $P^g_i$ by $P_i$ in $\mathbb P_i$.
\end{rem}

\begin{proof}
{\bf (1)}: In this case, the parabolic trees are either singletons stabilized by a conjugate of an element of $\mathbb P_1\cup \mathbb P_2-\{P_1\}$, or parabolic trees are $i$-pods stabilized by conjugates of $P_2$.

{\bf (2)}: The proof is similar.

{\bf (3)}: All parabolic trees are singletons except for those that are translates of a copy of the Bass-Serre tree for $P*_{P^t=P'}$. Following the proof of Theorem~\ref{first}, let $\nu \in \bar K$, if the preimage of $\nu$ in $K$ is not attached to an edge space, then $G_\nu$ is conjugate to an element of $\mathbb P-\{P\}$, otherwise $G_\nu$ is conjugate to $\langle P, t \rangle$.
\end{proof}

\begin{exmp}
We encourage the reader to consider the case of Theorem~\ref{first} and Corollaries~\ref{hypspli}~and~\ref{mlemma}, in the scenario where $G$ splits as a graph of free groups with cyclic edge groups.
A very simple case is:
Let $G=\langle a, b, t ~|~ (W^n)^t=W^m\rangle$ where $W\in \langle a, b \rangle$ and $m,n\geq 1$. Then $G$ is hyperbolic relative to $\langle W,t\rangle$.
\end{exmp}

\section{Relative Quasiconvexity}\label{sec:qc}
Dahmani introduced the notion of relatively quasiconvex subgroup in \cite{Dahmani03}.
This notion was further developed by Osin  in \cite{OsinBook06}, and later
Hruska  investigated several equivalent definitions of relatively quasiconvex subgroups \cite{HruskaRelQC}.
Martinez-Pedroza and the second author introduced a definition of relative quasiconvexity in the context of fine hyperbolic graphs and showed this definition is equivalent to Osin's definition \cite{MartinezPedrozaWise2011}. We will study relatively quasiconvexity
using this fine hyperbolic viewpoint.
Our aim is to examine the relative quasiconvexity of a certain subgroup which are themselves amalgams,
and we note that powerful results in this direction are given in \cite{MartinezPedrozaCombinations2009}.

\begin{defn}[Relatively Quasiconvex]\label{rqc}
Let $G$ be hyperbolic relative to $\mathbb{P}$. A subgroup $H$
of $G$ is \emph{quasiconvex relative to $\mathbb{P}$} if for some (and hence any) $(G; \mathbb{P})$-graph $K$, there is a nonempty connected and quasi-isometrically embedded, $H$-cocompact subgraph $L$ of $K$. In the sequel, we sometimes refer to $L$ as a \emph{quasiconvex $H$-cocompact subgraph} of $K$.
\end{defn}

\begin{rem}\label{fg for lqc}
It is immediate from the Definition~\ref{rqc} that in a relatively hyperbolic group, any parabolic subgroup is relatively quasiconvex, and any relatively quasiconvex subgroup is also relatively hyperbolic. In particular, the relatively quasiconvex subgroup $H$ is hyperbolic relative to the collection  $\mathbb P_H$ consisting of representatives of $H$-stabilizers of vertices of $L\subseteq K$. Note that a conjugate of a relatively quasiconvex subgroup is also relatively quasiconvex.
\begin{com} conjugate has easy proof, and intersection harder, but both obtainable from fine viewpoint. intersection needs reference\end{com}
And the intersection of two relatively quasiconvex subgroups is relatively quasiconvex.
Specifically, this last statement was proven when $G$ is f.g. in \cite{MartinezPedrozaCombinations2009},
and when $G$ is countable in \cite{HruskaRelQC}.
\end{rem}




Relative quasiconvexity has the following transitive property
proven by Hruska for countable relatively hyperbolic groups in \cite{HruskaRelQC}:

\begin{lem}\label{associative}
Let $G$ be hyperbolic relative to $\mathbb P_G$.
Suppose that $B$ is relatively quasiconvex in $G$,
and note that $B$ is then hyperbolic relative to $\mathbb P_B$ as in Remark~\ref{fg for lqc}.
Then $A\leq B$ is quasiconvex relative to $\mathbb P_B$ if and only if $A$ is quasiconvex relative to $\mathbb P_G$.
\end{lem}
\begin{proof}
Let $K$ be a $(G; \mathbb{P}_G)$-graph. As $B$ is quasiconvex relative to $\mathbb P_G$, there is a $B$-cocompact and quasiconvex subgraph $L\subset K$. Note that $L$ is a $(B; \mathbb{P}_B)$-graph. Let $A\leq B$.

If  $A$ is quasiconvex in $B$ relative to $\mathbb P_B$, there is an $A$-cocompact quasiconvex subgraph $M\subset L$. Since the composition $L_A\rightarrow L_B\rightarrow K$ is a quasi-isometric embedding, $A$ is quasiconvex relative to $\mathbb P_G$.
Conversely, if $A$ is quasiconvex in $G$ relative to $\mathbb P_G$, then there is an $A$-cocompact quasiconvex subgraph $M\subset K$. Let $L'=L\cup BM$ and note that $L'$ is $B$-cocompact and hence also quasiconvex,
and thus $L'$ also serves as a fine hyperbolic graph for $B$.
Now $M\subset L'$ is quasiconvex since $M\subset L$ is quasiconvex so $A$ is relatively quasiconvex in $B$.
\end{proof}

\begin{rem}\label{rem:G_v automatically quasiconvex}
One consequence of Theorem~\ref{first} and its various Corollaries,
is that when $G$ splits as a graph of relatively hyperbolic groups with parabolic subgroups,
then each of the vertex groups is quasiconvex relative to the peripheral structure of $G$.
(For Theorem~\ref{first} this is $\mathbb Q$, and for Corollary~\ref{hypspli} this is $\mathbb P  - \{\text{repeats}\}$.)
Indeed,  $K_v$ is a $G_v$-cocompact quasiconvex subgraph in
the fine graph $K$ constructed in the proof.
\end{rem}

\begin{lem}\label{fg}
Let $G$ be a f.g. group that split as a finite graph of groups $\Gamma$.
If each edge group is f.g. then each vertex group is f.g.
\end{lem}

\begin{proof}
Let $G=\langle g_1,\dots, g_n \rangle$. We regard $G$ as $\pi_1$ of a $2$-complex corresponding to $\Gamma$.
We show that  each vertex group  $G_v$ equals
  $\langle \{G_e\}_{\text{$e$ attached  to~$v$}} \cup \{ g\in G_v :  g~\text{in normal form of some}~g_i \}\rangle$.
Let $a\in G_v$ and consider an expression of $a$ as a product of normal forms
 of the $g_i^{\pm1}$. Then $a$ equals some product $a_1{t_1}^{\epsilon_1}b_1{t_2}^{\epsilon_2}a_2\cdots a_n{t_m}^{\epsilon_m}b_k$. There is a disc diagram $D$ whose boundary path is $a^{-1}a_1{t_1}^{\epsilon_1}b_1{t_2}^{\epsilon_2}a_2\cdots a_n{t_m}^{\epsilon_m}b_k$. See Figure~\ref{fig:DiscDiagram}. The region of $D$ that lies along $a$ shows that $a$ equals
the product of elements in edge groups adjacent to $G_v$, together with elements of $G_v$ that lie in the normal forms of $g_1,\dots, g_n$.
\end{proof}

\begin{figure}\centering
\includegraphics[width=.25\textwidth]{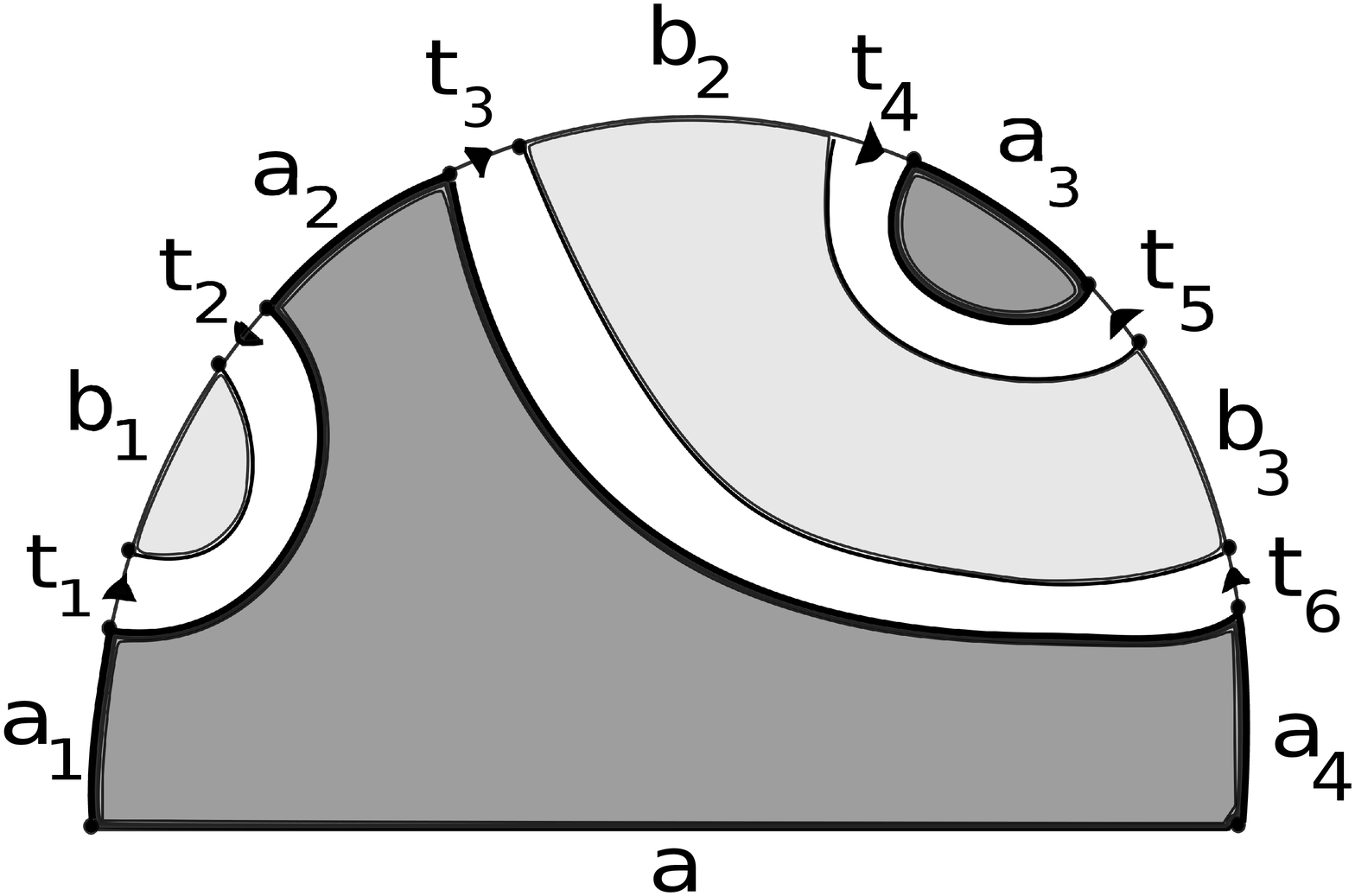}
\caption{
 \label{fig:DiscDiagram}}
\end{figure}

\begin{thm}[Quasiconvexity of a Subgroup in Parabolic Splitting]\label{qc first}
Let $G$ split as a finite graph $\Gamma$ of relatively hyperbolic groups such that each edge group is parabolic in its vertex groups. (Note that $G$ is hyperbolic relative to $\mathbb{Q} = \{Q_1,\dots,Q_j\}$
 by Theorem ~\ref{first}.) Let $H\leq G$ be tamely generated. 
Then $H$ is quasiconvex relative to $\mathbb Q$. Moreover if each $H_v$ in the Bass-Serre tree $T$ is finitely generated then $H$ is finitely generated.
\end{thm}
\begin{proof}
Since there are finitely many orbits of vertices whose stabilizers are finitely generated, $H$ is finitely generated.
For each $u\in \Gamma^0$, let $G_u$ be hyperbolic relative to $\mathbb P_u$ and let $K_u$ be a $(G_{u}; \mathbb P_{u})$-graph. Let $K$ be the $(G;\mathbb Q)$-graph constructed in the proof of Theorem~\ref{first} and let $\bar K$ be its quotient.
We will construct an $H$-cocompact quasiconvex, connected subgraph $\bar L$ of $\bar K$.

Let $T_H$ be the minimal $H$-invariant subgraph of $T$. Recall that each edge of $T$ (and hence $T_H$) corresponds to an edge of $K$. Let $F_H$ denote the subgraph of $K$ that is the union of all edges correspond to edges of $T_H$. Let $\{\nu_1,\dots,\nu_n\}$ be a representatives of $H$-orbits of vertices of $T_H$. For each $i$, let $L_i\hookrightarrow K_{\nu_i}$ be a $(H\cap G_{\nu_i}^{g_{i}})$-cocompact quasiconvex subgraph such that $L_i$ contains $F_H\cap K_{\nu_i}$. (There are finitely many $(H\cap G_{\nu_i}^{g_{i}})$-orbits of such endpoints of edges in $K_{\nu_i}$.) Let $L=F_H\cup \bigcup_{i=1}^{n} HL_i$ and let $\bar L$ be the image of $L$ under $K\rightarrow \bar K$. Observe that $L$ is quasiconvex in $K$ since $K$ is a ``tree union'' and each such $L_i$ of $L$ is quasiconvex in $K_{\nu_i}$. And likewise, $\bar L$ is quasiconvex in $\bar K$.
\begin{com}WE SHOULD ADD FIGURE\end{com}
%
\end{proof}

\begin{cor}[Characterizing Quasiconvexity in Maximal Parabolic Splitting]\label{qc in parab}
Let $G$ split as a finite graph of countable groups. For each $\nu$, let $G_\nu$ be hyperbolic relative to $\mathbb P_\nu$ and let the collection $\{G_e: e \textrm{ is attached to } \nu \}$ be a collection of maximal parabolic subgroups of $G_\nu$. $($Note that $G$ is hyperbolic relative to $\mathbb P=\bigcup_{\nu} \mathbb P_{\nu}-\{\text{repeats}\}$ by Corollary~\ref{hypspli}.$)$ Let $T$ be the Bass-Serre tree and let $H$ be a subgroup of $G$. The following are equivalent:
\begin{enumerate}
\item $H$ is tamely generated and each $H_v$ in the Bass-Serre tree $T$ is f.g.
\item $H$ is f.g. and quasiconvex relative to $\mathbb P$.
\end{enumerate}
\end{cor}

\begin{proof}
{\bf {(1 $\Rightarrow$ 2)}}: Follows from Theorem~\ref{first} and Theorem~\ref{qc first}.

{\bf {(2 $\Rightarrow$ 1)}}: Since $H$ is f.g., the minimal $H$-subtree $T_H$ is $H$-cocompact, and so $H$ splits as a finite graph of groups $\Gamma_H$. Since $H$ is quasiconvex in $\mathbb P$, it is hyperbolic relative to intersections with conjugates of $\mathbb P$. In particular, the infinite edge groups in the induced splitting of $H$ are maximal parabolic, and are thus f.g.
since the maximal parabolic subgroups of a f.g. relatively hyperbolic group are f.g. \cite{OsinBook06}.
 Each vertex group of $\Gamma_H$ is f.g. by Lemma~\ref{fg}.

By Remark~\ref{rem:G_v automatically quasiconvex}, each vertex group of $G$ is quasiconvex relative to $\mathbb P$,
and hence each $G_\nu$ is relatively quasiconvex by Remark~\ref{fg for lqc} since it is a conjugate of a vertex group.
Thus  $H_\nu= H\cap G_\nu$  is quasiconvex relative to $\mathbb P$ by Remark~\ref{fg for lqc}.
Finally, $H_\nu$ is quasiconvex in $G_\nu$ by Lemma~\ref{associative}.
\end{proof}

\section{Local Relative Quasiconvexity}\label{sec:lqc}


A relatively hyperbolic group $G$ is \emph{locally relatively quasiconvex} if each f.g. subgroup of $G$ is relatively quasiconvex.
The focus of this section is the following criterion for showing that the combination of locally relatively quasiconvex groups
is again locally relatively quasiconvex.\begin{com}we should maybe say we mean finite combination\end{com}

Recall that $N$ is \emph{Noetherian} if each subgroup of $N$ is f.g. We now  give a criterion for local quasiconvexity
 of a  group that splits along parabolic subgroups.

\begin{thm}[A Criterion for Locally Relatively Quasiconvexity]\label{intersect}
\begin{enumerate}
\item Let $G_1$ and $G_2$ be locally relatively quasiconvex relative to $\mathbb P_1$ and $\mathbb P_2$. Let $G=G_1 \ast_{P_1={P_2'}} G_2$ where each $P_i\in \mathbb P_i$ \begin{com}this is not very well stated\end{com}and $P_1$ is identified with the subgroup ${P_2}'$ of $P_2$. Suppose $P_1$ is Noetherian. Then $G$ is locally quasiconvex relative to $\mathbb P_1\cup \mathbb P_2-\{P_1\}$.
\item  Let $G_1$ be a locally relatively quasiconvex relative to $\mathbb P$. Let $P_1\in \mathbb P$ be isomorphic to a subgroup ${P_2}'$ of a maximal parabolic subgroup $P_2$ not conjugate to $P_1$. Let $G = G_1*_{{P_1}^t={P_2}'}$. Suppose $P_1$ is Noetherian. Then $G$ is
locally quasiconvex relative to $\mathbb P-\{P_1\}$.
\item Let $G_1$ be a locally quasiconvex relative to $\mathbb P$. Let $P$ be a maximal parabolic
subgroup of $G_1$, isomorphic to ${P}'\leq P$. Let $G = G_1*_{{P}^t={P}'}$ and suppose $P$ is Noetherian. Then $G$ is
also locally quasiconvex relative to $\mathbb P\cup \langle P, t \rangle-\{P\}$.
\end{enumerate}
\end{thm}
\begin{proof}
(1): By Corollary~\ref{mlemma}, $G$ is hyperbolic relative to $\mathbb P=\mathbb P_1\cup \mathbb P_2-\{P_1\}$. Let $H$ be a finitely generated subgroup of $G$. We show that $H$ is quasiconvex relative to $\mathbb P$. Let $T$ be the Bass-Serre tree of $G$. Since $H$ is f.g., the minimal $H$-subtree $T_H$ is $H$-cocompact, and so $H$ splits as a finite graph of groups $\Gamma_H$.
Moreover, the edge groups of this splitting are f.g. since the edge groups of $G$ are Noetherian by hypothesis.
Thus each vertex group of $\Gamma_H$ is f.g. by Lemma~\ref{fg}.
 Since $G_1$ and $G_2$ are locally relatively quasiconvex, each vertex group of $T_H$ is relatively quasiconvex in its ``image vertex group'' under the map $T_H\rightarrow T$. Now by Theorem~\ref{qc first}, $H$ is quasiconvex relative to $\mathbb P$. The proof of (2) and (3) are similar.
\end{proof}

\begin{defn}[Almost Malnormal]
A subgroup $H$ is \emph{malnormal} in $G$ if $H\cap H^g=\{1\}$ for $g\notin H$,
 and similarly $H$ is \emph{almost malnormal} if this intersection $H\cap H^g$ is always finite.
 Likewise, a collection of subgroups $\{H_i\}$ is \emph{almost malnormal}
 if $H_i^g\cap H_j^h$ is finite unless $i=j$ and $gh^{-1}\in H_i$.
\end{defn}

\begin{cor}\label{cor:lqc noetherian combination}
Let $G$ split as a finite graph of groups. Suppose
\begin{itemize}
\item [a)] Each $G_{\nu}$ is locally relatively quasiconvex;
\item[b)] Each $G_e$ is Noetherian and maximal parabolic in its vertex groups;
\item[c)] $\{G_e: e \textrm{  is attached to } \nu\}$ is almost malnormal in $G_{\nu}$, for any vertex $\nu$.
\end{itemize}
Then $G$ is locally relatively quasiconvex relative to $\mathbb P$, see Corollary~\ref{hypspli}.
\end{cor}

\subsection{Small-hierarchies and local quasiconvexity}\label{Noetherian}
The main result in this subsection is a consequence of Theorem~\ref{intersect}
that employs results of Yang \cite{Yang11} stated in Theorems~\ref{Yang}~and~\ref{hyp extend},
 and also depends on Lemma~\ref{QC vertex} which is independent of Section~\ref{sec:improvements using yang}. The reader may choose to read this subsection and refer ahead to those results, or return to this subsection
 after reading Section~\ref{sec:improvements using yang}.

\begin{defn}[Small-Hierarchy]\label{defn:small}
A group is \emph{small} if it has no rank~2 free subgroup.
Any small group has a {\em{length~0 small-hierarchy}}. $G$ has a {\em{length~$n$ small-hierarchy}} if $G\cong A*_CB$ or
$G\cong A*_{C^t=C'}$, where $A$ and $B$ have length $(n-1)$ small-hierarchies, and $C$ is small and f.g.
\begin{com} note that if $G$ has a length $n$ hierarchy, then it also has a length $n+m$ hierarchy for $m\geq 1$,
since we can repeatedly take phony amalgamations with trivial trivial groups $G*_1 1$.\end{com}
We say $G$ has a \emph{small-hierarchy} if it has a length~$n$ small-hierarchy for some $n$.

\begin{com}Let $\mathcal S$ be a class of small groups that is closed under isomorphism and subgroups.
Some typical such classes are: Abelian groups, Polycyclic groups, Finite groups, Noetherian groups, etc.\end{com}

We can define $\mathcal F$-hierarchy by replacing ``small'' by a class of groups $\mathcal F$ closed under subgroups and isomorphisms. For instance, when $\mathcal F$ is the class of finite groups, the class of groups with an $\mathcal F$-hierarchy
is precisely the class of virtually free groups.
\end{defn}

\begin{rem}\label{Tits Alternative}
The \emph{Tits alternative} for relatively hyperbolic groups states that every f.g. subgroup is either:
 elementary, parabolic, or contains a subgroup isomorphic to $F_2$.
The Tits alternative is proven for countable relatively hyperbolic groups in \cite[Thm~8.2.F]{Gromov87}.
A proof is given for convergence groups in \cite{Tukia94}.
It is shown in \cite{OsinBook06} that every cyclic subgroup $H$ of a f.g. relatively hyperbolic group $G$ is relatively quasiconvex.
\end{rem}

\begin{com}
Can a relatively hyperbolic group contain an infinite torsion subgroup that is not parabolic?
if so, then we would not have to assume f.g. edge groups in the definition of small-hierarchy.

Is the following theorem more general if we assume that the small-hierarchy has quasconvex vertex groups at each stage.
\end{com}
\begin{thm}\label{thm:Noetherian}
Let $G$ be f.g. and hyperbolic relative to $\mathbb P$ where each element of\ $\mathbb P$ is Noetherian.
Suppose $G$ has a small-hierarchy. Then $G$ is locally relatively quasiconvex.
\end{thm}


\begin{proof}
The proof is by induction on the length of the hierarchy.
Since edge groups are f.g., the Tits alternative shows that there are three cases according to whether the edge group is finite, virtually cyclic, or infinite parabolic, and we note that the edge group is relatively quasiconvex in each case.
These three cases are each divided into two subcases according to whether $G=A*_{C_1} B$ or $G=A*_{C_1^t=C_2}$.

Since $C_1$ and $G$ are f.g. the vertex groups are f.g. by Lemma~\ref{fg}.
Thus, since $C_1$ is relatively quasiconvex the vertex groups are relatively quasiconvex by Lemma~\ref{QC vertex}.

When $C_1$ is finite the conclusion follows in each subcase from Theorem~\ref{intersect}.

When $C_1$ is virtually cyclic but not parabolic, then $C_1$ lies in a unique maximal virtually cyclic subgroup $Z$ that is almost malnormal and relatively quasiconvex by \cite{Osin06}.\begin{com} He proved that for any hyperbolic $g$, there is an elementary $C_1$ such that $g \in C_1$\end{com}
Thus $G$ is hyperbolic relative to $\mathbb P' = \mathbb P \cup \{Z\}$ by Theorem~\ref{hyp extend}.

Observe that $C_1$ is maximal infinite cyclic on at least one side, since otherwise there would be a nontrivial splitting of $Z$ as an amalgamated free product over $C_1$. We equip the (relatively quasiconvex) vertex groups with their induced peripheral structures. Note that $C_1$ is maximal parabolic on at least one side and so $G$ is locally relatively quasiconvex relative to $\mathbb P'$ by Theorem~\ref{intersect}.
Finally, by Theorem~\ref{Yang}, any subgroup $H$ is quasiconvex relative to the original peripheral structure $\mathbb P$ since intersections between $H$ and conjugates of $Z$ are quasiconvex relative to $\mathbb P$.

When $C_1$ is infinite parabolic, we will first produce a new splitting before verifying local relative quasiconvexity.

When $G=A*_{C_1} B$.
Let $D_a, D_b$ be the maximal parabolic subgroups of $A,B$ containing $C_1$,
and refine the splitting to:
$$ A*_{D_a}  ( D_a*_{C_1} D_b) *_{D_b} B$$
The two outer splittings are along a  parabolic that is maximal on the outside vertex group.
The inner vertex group $D_a*_{C_1} D_b$ is a single parabolic subgroup of $G$. Indeed, as $C_1$ is infinite,
 $D_a \supset {C_1} \subset D_b$ must all lie in the same parabolic subgroup of $G$. It is obvious that $ D_a*_{C_1} D_b$ is locally relatively quasiconvex with respect to its
induced peripheral structure since it is itself parabolic in $G$. Consequently $( D_a*_{C_1} D_b) *_{D_b} B$ 
is locally relatively quasiconvex by Theorem~\ref{intersect}, therefore $G=A*_{D_a}  \big(( D_a*_{C_1} D_b) *_{D_b} B\big)$ is 
locally relatively quasiconvex by Theorem~\ref{intersect}.

When $G\cong A*_{{C_1}^t=C_2}$, let $M_i$ be the maximal parabolic subgroup of $G$ containing $C_i$.
 There are two subsubcases:

{\bf{\large[}$t \in M_1${\large ]}} Then $C_2 \leq M_1$ and
 \begin{com} Note that since $M_1$ cannot contain $F_2$, the group $M_1$ is an ascending HNN extension, and would be a semi-direct product in the Noetherian case.(need explanation) and\end{com}
 we revise the splitting to $G\cong A*_{D_1} M_1$ where $D_1 = M_1\cap A$.
And in this splitting the edge group is maximal parabolic at $D_1\subset A$, and $M_1$ is parabolic.

{\bf{\large[}$t \notin M_1${\large ]}}
\begin{com}Let $M_1=M$ and $M_2=M^t$ be the maximal parabolic subgroup of $G$ containing $C_2$.
Note that $\{M_1,M_2\}$ is almost malnormal since $t\not\in M_1$.\end{com}
Let $D_i$ denote the maximal parabolic subgroup of $A$ containing $C_i$.
Observe that $\{D_1,D_2\}$  is almost malnormal since $D_i=M_i\cap A$.
We revise the HNN extension to the following:
$$\bigg(D_1^t*_{C_1^t= C_2} A \bigg)*_{{D_1}^t=D_1}$$
where the conjugated copies of $D_1$ in the HNN extension
embed in the first  and second factor of the AFP.

In both cases, the local relative quasiconvexity of $G$ now holds by Theorem~\ref{intersect} as before.
 \end{proof}

\section{Relative Quasiconvexity in Graphs of Groups}\label{sec:improvements using yang}

Gersten \cite{Gersten96Subgroups} and then Bowditch \cite{Bowditch99RH} showed that a hyperbolic group $G$ is hyperbolic relative to an almost malnormal quasiconvex subgroup.  Generalizing work of Martinez-Pedroza \cite{MartinezPedroza08}, Yang introduced and characterized a class of parabolically extended structures for countable relatively hyperbolic groups
\cite{Yang11}. We use his results to generalize our previous results. The following was defined in \cite{Yang11} for countable groups.

\begin{defn}[Extended Peripheral Structure]\label{parabolic structure}
A {\em{peripheral structure}} consists of a finite collection  $\mathbb P$ of subgroups of a group $G$.
Each element $P\in \mathbb P$ is a {\em{peripheral subgroup}} of $G$.
The peripheral structure $\mathbb E = \{E_j\}_{j\in J}$  \emph{extends}
$\mathbb P = \{P_i\}_{i\in I}$ if for each $i \in I$, there exists $j \in J$ such that $P_i \subseteq E_j$.
For  $E\in \mathbb E$, we let $\mathbb P_{E} = \{P_i \, : \,  P_i\subseteq E, P_i\in \mathbb P, i\in I\}$.
\end{defn}



We will use the following result of Yang~\cite{Yang11}.

\begin{thm}[Hyperbolicity of Extended Peripheral Structure]\label{hyp extend}
Let $G$ be hyperbolic relative to $\mathbb P$ and let the peripheral structure $\mathbb E$ extend $\mathbb P$.
Then $G$ is hyperbolic relative to $\mathbb E$ if and only if the following hold:
\begin{enumerate}
\item \label{yang:2} $\mathbb E$ is almost malnormal;
\item \label{yang:1}  Each $E\in \mathbb E$ is quasiconvex in $G$ relative to $\mathbb P$.
\end{enumerate}
\end{thm}

\begin{defn}[Total]
Let $G$ be hyperbolic relative to $\mathbb P$. The subgroup $H$ of $G$ is {\em{total relative to $\mathbb P$}} if: either $H\cap P^g=P^g$ or $H\cap P^g$ is finite for each $P\in \mathbb P$ and $g\in G$.
\end{defn}

The following is proven in \cite{DrutuSapir2005}:
\begin{com}If a big space is a tree graded space of medium spaces each of which is a tree graded space of small spaces,
then the big space is a tree graded space of small spaces\end{com}
\begin{lem}\label{associative hyp}
If $G$ is f.g. and hyperbolic relative to $\mathbb P=\{P_1,\dots, P_n\}$ and
each $P_i$ is hyperbolic relative to $\mathbb H_i=\{H_{i1},\dots,H_{i{m_i}}\}$, then $G$ is hyperbolic relative to $\bigcup_{1\leq i\leq n} \mathbb H_{i}$.
\end{lem}
As an application of Theorem~\ref{hyp extend}, we now generalize Corollary~\ref{mlemma} to handle the case where edge groups are quasiconvex and not merely parabolic.
\begin{thm}[Combination along Total, Malnormal and Quasiconvex Subgroups] \label{QC combination}
\indent
\begin{enumerate}
\item \label{aaa:1} Let $G_i$ be hyperbolic relative to $\mathbb P_i$ for $i=1, 2$. Let $C_i \leq G_i$ be almost malnormal, total and relatively quasiconvex. Let ${C_1}'\leq C_1$. Then $G=G_1*_{{C_1}'=C_2} G_2$ is hyperbolic relative to $\mathbb P=\mathbb P_1 \cup \mathbb P_2-\{P_2\in \mathbb P_2~:~P^g_2\subseteq C_2,~\text{for some}~g\in G_2\}$.
\item \label{bbb:1} Let $G_1$ be hyperbolic relative to $\mathbb P$. Let $\{C_1,C_2\}$ be almost malnormal and assume each $C_i$ is total and relatively quasiconvex. Let ${C_1}'\leq C_1$. Then $G=G_1*_{{{C_1}'={C_2}^t}}$ is hyperbolic relative to $\mathbb P=\mathbb P-\{P_2\in \mathbb P_2~:~P^g_2\subseteq C_2,~\text{for some}~g\in G_2\}$.
\end{enumerate}
\end{thm}

\begin{proof}

{\bf (1)}: For each $i$, let
$$\mathbb E_i=\mathbb P_i-\{P\in \mathbb P_i \colon P^g\leq {C_i},~\text{for some}~ g\in G_i\}\cup \{C_i\}$$ Without loss of generality, we can assume that $\mathbb E_i$ extends $\mathbb P_i$, since we can replace an element of $\mathbb P_i$ by its conjugate. We now show that $G_i$ is hyperbolic relative to $\mathbb E_i$ by verifying the two conditions of Theorem~\ref{hyp extend}: $\mathbb E_i$ is malnormal in $G_i$, since $\mathbb P_i$ is almost malnormal and $C_i$ is total and almost malnormal. Each element of $\mathbb E_i$ is relatively quasiconvex, since $C_i$ is relatively quasiconvex by hypothesis and each element of $\mathbb P_i$ is relatively quasiconvex by Remark~\ref{fg for lqc}.

We now regard each $G_i$ as hyperbolic relative to $\mathbb E_i$. Therefore since the edge group $C_2={C_1}'$ is maximal on one side, by Corollary~\ref{mlemma}, $G$ is hyperbolic relative to $\mathbb E=\mathbb E_1\cup \mathbb E_2-\{C_2\}$.

We now apply Lemma~\ref{associative hyp} to show that $G$ is hyperbolic relative to $\mathbb P$. We showed that $G$ is hyperbolic relative to $\mathbb E$. But each element of $\mathbb E$ is hyperbolic relative to $\mathbb P$ that it contains. Thus by Lemma~\ref{associative hyp}, we obtain the result.

{\bf {(2)}}: The proof is analogous to the proof of $(1)$.
\end{proof}
The following can be obtained by induction using Theorem~\ref{QC combination} or can be proven directly using the same mode of proof.
\begin{cor}\label{hyp Qc combination}
Let $G$ split as a finite graph of groups. Suppose
\begin{itemize}
\item[(a)] Each $G_{\nu}$ is hyperbolic relative to $\mathbb P_{\nu}$;
\item[(b)] Each $G_e$ is total and relatively quasiconvex in $G_{\nu}$;
\item[(c)] $\{G_e: e~ \textrm{is attached to } \nu\}$ is almost malnormal in $G_{\nu}$ for each vertex $\nu$.
\end{itemize}

Then $G$ is hyperbolic relative to $\bigcup_{\nu} \mathbb P_{\nu}-\{\text{repeats}\}$.
\end{cor}

Yang characterized relative quasiconvexity
with respect to extensions in \cite{Yang11} as follows:

\begin{thm}[Quasiconvexity in Extended Peripheral Structure]\label{Yang}
Let $G$ be hyperbolic relative to $\mathbb P$ and relative to $\mathbb E$. Suppose that $\mathbb E$ extends $\mathbb P$.
Then
\begin{enumerate}
\item If $H\leq G$ is quasiconvex relative to $\mathbb P$,
then $H$ is quasiconvex relative to $\mathbb E$.

\item Conversely, if $H\leq G$ is quasiconvex relative to $\mathbb E$, then $H$ is quasiconvex relative to $\mathbb P$ if and only if $H\cap {E}^{g}$ is quasiconvex relative to $\mathbb P$ for all $g \in G$ and $E \in \mathbb E$.
\end{enumerate}
\end{thm}

We recall the following observation of Bowditch (see~\cite[Lem~2.7 and 2.9]{MartinezPedrozaWise2011}).
\begin{lem}[$G$-attachment]
Let $G$ act on a graph $K$. Let $p,q \in K^0$ and $e$ be a new edge whose endpoints are $p$ and $q$. The \emph{$G$-attachment} of $e$ is the new graph $K' = K \cup Ge$ which consists of the union of $K$ and copies $ge$ of $e$ attached at $gp$ and $gq$ for any $g\in G$. Note that $K'$ is $G$-cocompact/fine/hyperbolic if $K$ is.\begin{com}G-cofinite\end{com}
\end{lem}

In the following lemma, we prove that when a relatively hyperbolic group $G$ splits then relative quasiconvexity of vertex groups is equivalent to relative quasiconvexity of the edge groups.

\begin{lem}[Quasiconvex Edges $\Longleftrightarrow$ Quasiconvex Vertices]\label{QC vertex}
Let $G$ be hyperbolic relative to $\mathbb P$. Suppose $G$ splits as a finite graph of groups whose vertex groups and edge groups are finitely generated. Then the edge groups are quasiconvex relative to $\mathbb P$ if and only if the vertex groups are quasiconvex relative to $\mathbb P$.
\end{lem}
\begin{proof}
If the vertex groups are quasiconvex relative to $\mathbb P$ then so are the edge groups, since relative quasiconvexity is preserved by intersection (see \cite{HruskaRelQC,MartinezPedrozaCombinations2009}) in the f.g. group $G$.
Assume the edge groups are quasiconvex relative to $\mathbb P$. Let $K$ be a $(G;\mathbb P)$ graph and let $T$ be the Bass-Serre tree for $G$. Let
 $f \colon K \rightarrow T$ be a $G$-equivariant map that sends vertices to vertices and edges to geodesics. Subdivide $K$ and $T$, so that each edge is the union of two length $\frac{1}{2}$~\emph{halfedges}. Let $\nu$ be a vertex in $T$. It suffices to find a $G_\nu$-cocompact quasiconvex subgraph $L$ of $K$.

Let $\{e_1,\dots,e_m\}$ be representatives of the $G_\nu$-orbits of halfedges attached to $\nu$. Let $\omega_i$ be the other vertex of $e_i$ for $1\leq i \leq m$. Since each $G_{\omega_i}=G_{e_i}$ is f.g. by hypothesis, we can perform finitely many $G_{\omega_i}$-attachments of arcs so that the preimage of $\omega_i$ is connected for each $i$. This leads to finitely many $G$-attachments to $K$ to obtain a new fine hyperbolic graph $K'$. By mapping the newly attached edges to their associated vertices in $T$, we thus obtain a $G$-equivariant map $f' \colon K'\rightarrow T$ such that $M'_i=f'^{-1} (\omega_i)$ is connected and $G_{\omega_i}$-cocompact for each $i$.

Consider $L'=f'^{-1}(N_{\frac{1}{2}}(\nu))$ where $N_{\frac{1}{2}}(\nu)$ is the closed $\frac{1}{2}$-neighborhood of $\nu$. To see that $L'$ is connected, consider a path $\sigma$ in $K'$ between distinct components of $L'$. Moreover choose $\sigma$ so that its image in $T$ is minimal among all such choices. Then $\sigma$ must leave and enter $L'$ through the same $g_\nu M'_i$ which is connected by construction.

We now show that $L'$ is quasiconvex. Consider a geodesic $\gamma$ that intersects $L'$ exactly at its endpoints. As before the endpoints of $\gamma$ lie in the same $g_\nu M'_i$. Since $g_\nu M'_i$ is $\kappa_i$-quasiconvex for some $\kappa_i$, we see that $\gamma$ lies in $\kappa$-neighborhood of $g_\nu M'_i$ and hence in the $\kappa$-neighborhood of $L'$.
\end{proof}

\begin{lem}[Total Edges $\Longleftrightarrow$ Total Vertices]\label{total}
Let $G$ be hyperbolic relative to $\mathbb P$. Let $G$ act on a tree $T$. For each $P\in \mathbb P$ let $T_P$ be a minimal $P$-subtree. Assume that no $T_P$ has a finite edge stabilizer in the $P$-action. Then edge groups of $T$ are total in $G$ iff
  vertex groups are total in $G$.
\end{lem}
\begin{proof}
Since the intersection of two total subgroups is total, if the vertex groups are total then the edge groups are also total. We now assume that the edge groups are total. Let $G_{\nu}$ be a vertex group and $P\in \mathbb P$ such that $P^g\cap G_\nu$ is infinite for some $g\in G$. If $|P^g\cap G_e|=\infty$ for some edge $e$ attached to $\nu$, then $P\subseteq G_e$, thus $P\subseteq G_e\subseteq G_\nu$. Now suppose that $|P^g\cap G_e|<\infty$ for each $e$ attached to $\nu$. If $P^g\nleq G_\nu$ then the action of $P^g$ on $gT$ violates our hypothesis.
\end{proof}

\begin{rem}\label{parab spli}
Suppose $G$ is f.g. and $G$ is hyperbolic relative to $\mathbb P$. Let $P\in \mathbb P$ such that $P=A*_C B$ [$P=A*_{C={C'}^t}$] where $C$ is a finite group. Since $P$ is hyperbolic relative to $\{A, B\}$ [$\{A\}$], by Lemma~\ref{associative hyp}, $G$ is hyperbolic relative to $\mathbb P'=\mathbb P-\{P\}\cup \{A, B\}$ [$\mathbb P'=\mathbb P-\{P\}\cup \{A\}$].
\end{rem}

We now describe a more general criterion for relative quasiconvexity which is proven by combining Corollary~\ref{qc in parab} with Theorem~\ref{Yang}.
\begin{thm}\label{last}
Let $G$ be f.g. and hyperbolic relative to $\mathbb P$. Suppose $G$ splits as a finite graph of groups. Suppose
\begin{enumerate}[(a)]
\item\label{cond:aaa} Each $G_e$ is total in $G$;
\item\label{cond:bbb} Each $G_e$ is relatively quasiconvex in $G$;
\item\label{cond:ccc} $\{G_e: e \textrm{  is attached to } \nu\}$ is almost malnormal in $G_{\nu}$ for each vertex $\nu$.
\end{enumerate}
Let $H \leq G$ be tamely generated subgroup of $G$. Then $H$ is relatively quasiconvex in $G$.
\end{thm}
\begin{proof}
\emph{Technical Point:} By splitting certain elements of $\mathbb P$ to obtain $\mathbb P'$ as in Remark~\ref{parab spli}, we can assume that $G$ is hyperbolic relative to $\mathbb P'$ and each $G_\nu$ is hyperbolic relative to the conjugates of elements of $\mathbb P'$ that it contains.

Indeed for any $P\in \mathbb P$, if the action of $P$ on a minimal subtree $T_P$ of the Bass-Serre tree $T$, yields a finite graph $\Gamma$ of groups some of whose edge groups are finite, then following Remark~\ref{parab spli}, we can replace $\mathbb P$ by the groups that complement these finite edge groups, (i.e. the fundamental groups of the subgraphs obtained by deleting these edges from $\Gamma$.) Therefore $G$ is hyperbolic relative to $\mathbb P'$.
\begin{com} this process doesn't require f.p. of $P$ (and hence accessibility) to terminate.
Instead, it terminates since we are using the same tree- and hence finitely many possible edges where it can split\end{com}

No $P\in \mathbb P'$ has a nontrivial induced splitting as a graph of groups with a finite edge group.\begin{com}show to Dani\end{com}The edge groups are total relative to $\mathbb P'$ since they are total relative to $\mathbb P$. Therefore by Lemma~\ref{total} the vertex groups are total in $G$ relative to $\mathbb P'$. By Lemma~\ref{QC vertex}, each vertex group $G_\nu$ is relatively quasiconvex in $G$ relative to $\mathbb P$, therefore by Theorem~\ref{Yang} each $G_{\nu}$ is quasiconvex in $G$ relative to $\mathbb P'$. Thus $G_\nu$ has an induced relatively hyperbolic structure $\mathbb P'_\nu$ as in Remark~\ref{fg for lqc}. By totality of $G_\nu$, we can assume each element of $\mathbb P'_\nu$ is a conjugate of an element of $\mathbb P'$. And as usual we may omit the finite subgroups in $\mathbb P'_\nu$.

\emph{Step 1:} We now extend the peripheral structure of each $G_\nu$ from $\mathbb P'_\nu$ to $\mathbb E_\nu$ where $$\mathbb E_\nu=\{G_e: e\textrm{ is attached to }\nu\}\cup \{P\in \mathbb P'_{\nu}: P^g\nleq G_e \textrm{ for any $g\in G_\nu$}\}$$
\begin{com}show to dani\end{com}Almost malnormality of $\mathbb E_\nu$ follows from Condition~\eqref{cond:ccc} and the totality of the edge groups in their vertex groups which follows by the totality of the edge groups in $G$, also relative quasiconvexity of the new elements $G_e$ is Condition~\eqref{cond:bbb}.
Thus by  $G_\nu$ is hyperbolic relative to $\mathbb E_\nu$ by Theorem~\ref{Yang}.

\emph{Step 2:} For each $\tilde \nu$ in the Bass-Serre tree, its $H$-stabilizer $H_{\tilde \nu}$ lies in $G_{\tilde \nu}$ which
we identify
(by a conjugacy isomorphism) with the chosen vertex stabilizer $G_\nu$ in the graph of group decomposition.
Then $H_{\tilde{\nu}}$ is quasiconvex in $G_\nu$ relative to $\mathbb E_\nu$ for each $\nu$ by Theorem~\ref{Yang}, since $\mathbb E_\nu$ extends $\mathbb P'_\nu$ and each $H_{\tilde{\nu}}$ is quasiconvex in $G_\nu$ relative to $\mathbb P'_\nu$. Therefore $H$ is quasiconvex relative to $\bigcup \mathbb E_\nu$ by Corollary~\ref{qc in parab}.

\emph{Step 3:}  $H$ is quasiconvex relative to $\mathbb P'=\bigcup \mathbb P'_\nu$. Since $\bigcup \mathbb E_\nu$ extends $\mathbb P=\bigcup \mathbb P'_\nu$, by Theorem~\ref{Yang}, it suffices to show that $H\cap K^g$ is quasiconvex relative to $\mathbb P'$ for all $K\in \bigcup \mathbb E_\nu$ and $g\in G$. There are two cases:

Case 1: $K\in \mathbb P'_{\nu}$ for some $\nu$. Now $H\cap K^g$ is a parabolic subgroup of $G$ relative to $\mathbb P'$ and is thus quasiconvex relative to $\mathbb P'$.

Case 2: $K=G_e$ for some $e$ attached to some $\nu$. The group $K$ is relatively quasiconvex in $G_\nu$, therefore by Remark~\ref{fg for lqc}, $K^g$ is also relatively quasiconvex but in $G_{g\nu}$. Now since $K^g\cap H= K^g\cap H_{g\nu}$ and $K^g$ and $H_{g\nu}$ are both relatively quasiconvex in $G_{g\nu}$, the group $K^g\cap H$ is relatively quasiconvex in $G_{g\nu}$. Since by Lemma~\ref{QC vertex}, $G_{g\nu}$ is quasiconvex relative to $\mathbb P'$, Lemma~\ref{associative} implies that $K^g\cap H$ is quasiconvex relative to $\mathbb P'$.

Now $H$ is quasiconvex relative to $\mathbb P$ by Theorem~\ref{Yang}, since $\mathbb P$ extends $\mathbb P'$.
\end{proof}
The following result strengthens Theorem~\ref{last}, by relaxing Condition~\eqref{cond:ccc}.
\begin{thm}[Quasiconvexity Criterion for Relatively Hyperbolic Groups that Split]
\label{thm:strongest result}
Let $G$ be f.g. and  hyperbolic relative to $\mathbb P$ such that $G$ splits as a finite graph of groups.
 Suppose
\begin{enumerate}[(a)]
\item Each $G_e$ is total in $G$;
\item Each $G_e$ is relatively quasiconvex in $G$;
\item\label{cond:cc} Each $G_e$ is almost malnormal in $G$.
\end{enumerate}
Let $H\leq G$ be tamely generated. Then $H$ is relatively quasiconvex in $G$.
\end{thm}
\begin{rem}
By Lemma~\ref{associative} and Remark~\ref{rem:G_v automatically quasiconvex},
 Condition~\eqref{cond:bbb} is equivalent to requiring
that each $G_e$ is quasiconvex in $G_\nu$. Also we can replace Condition~\eqref{cond:aaa} by requiring $G_e$ to be total in $G_\nu$.
\end{rem}
\begin{proof}
 We prove the result by induction on the number of edges of the graph of groups $\Gamma$.
 The base case where $\Gamma$ has no edge is contained in the hypothesis.
  Suppose that $\Gamma$ has at least one edge $e$ (regarded as an open edge).
   If $e$ is nonseparating, then $G=A*_{C^t=D}$ where $A$ is the graph of groups over $\Gamma-e$,
   and $C,D$ are the two images of $G_e$.
 Condition~\eqref{cond:cc} ensures that $\{C, D\}$ is almost malnormal in $A$,
 and by induction, the various nontrivial intersections $H\cap A^g$ are relatively quasiconvex in $A^g$,
and thus $H$ is relatively quasiconvex in $G$ by Theorem~\ref{last}.
A similar argument concludes the separating case.
\end{proof}
\begin{cor}
Let $G$ be f.g. and hyperbolic relative to $\mathbb P$. Suppose $G$ splits as a finite graph of groups. Assume:
\begin{itemize}
\item[(a)] Each $G_{\nu}$ is locally relatively quasiconvex;
\item[(b)] Each $G_e$ is Noetherian, total and relatively quasiconvex in $G$;
\item[(c)] Each $G_e$ is almost malnormal in $G$.
\end{itemize}
Then $G$ is locally relatively quasiconvex relative to $\mathbb P$.
\end{cor}

\begin{thm}\label{thm:qc with parabolic vertices}
Let $G$ be hyperbolic relative to $\mathbb P$.
Suppose $G$ splits as a graph $\Gamma$ of groups with relatively quasiconvex edge groups.
Suppose $\Gamma$ is bipartite with $\Gamma^0=V\sqcup U$ and each edge joins vertices of $V$ and $U$.
Suppose each $G_v$ is maximal parabolic for $v\in V$, and for each $P\in \mathbb P$ there is at most one $v$
with $P$ conjugate to $G_v$.
Let $H\leq G$ be tamely generated. Then $H$ is quasiconvex relative to $\mathbb P$.
\end{thm}
The scenario of Theorem~\ref{thm:qc with parabolic vertices} arises when $M$ is a compact aspherical 3-manifold,
from its JSJ decomposition. The manifold $M$ decomposes as a bipartite graph $\Gamma$ of spaces with $\Gamma^0=U\sqcup V$.
The submanifold $M_v$ is  hyperbolic for each $v\in V$, and $M_u$ is a graph manifold for each $u\in U$.
The edges of $\Gamma$ correspond to the ``transitional tori'' between these hyperbolic and complementary graph manifold parts.
Some of the graph manifolds are complex but others are simpler Seifert fibered spaces; in the simplest cases, thickened tori between adjacent hyperbolic parts or $I$-bundles over Klein bottles where a hyperbolic part terminates.
Hence $\pi_1M$ decomposes accordingly as a graph $\Gamma$ of groups,
and $\pi_1M$ is hyperbolic relative to $\{\pi_1M_u: u\in U\}$ by Theorem~\ref{first} or indeed, Corollary~\ref{cor:first}.

\begin{proof}
Let $K_o$ be a fine hyperbolic graph for $G$.
Each vertex group is quasiconvex in $G$ by Lemma~\ref{QC vertex}, and
so for each $u\in U$ let $K_u$ be a $G_u$-quasiconvex subgraph, and in this way we obtain finite hyperbolic $G_u$-graphs,
and for $v\in V$, we let $K_v$ be a singleton.
We apply the Construction in the proof of Theorem~\ref{first} to obtain a fine hyperbolic $G$-graph $K$ and quotient $\bar K$.
Note that the parabolic trees are $i$-pods. We form the $H$-cocompact quasiconvex subgraph $L$
by combining $H_\omega$-cocompact quasiconvex subgraphs $K_\omega$ as in the proof of Theorem~\ref{qc first}.
\end{proof}

\begin{thm}\label{thm:qc with half malnormal vertices}
Let $G$ be f.g. and hyperbolic relative to $\mathbb P$.
Suppose $G$ splits as graph $\Gamma$ of groups with relatively quasiconvex edge groups.
Suppose $\Gamma$ is bipartite with $\Gamma^0=V\sqcup U$ and each edge joins vertices of $V$ and $U$.
Suppose each $G_v$ is almost malnormal and total in $G$ for $v\in V$.
\begin{com} It follows that each $G_e$ is total in $G_u$.
However, unless we assume that $\{G_v: v\in V\}$ is almost malnormal, it does not follow that: $\{G_e: e \textrm{  is attached to } u\}$ is almost malnormal for $u\in U$.\end{com}
Let $H\leq G$ be tamely generated. Then $H$ is quasiconvex relative to $\mathbb P$.
\end{thm}
Theorem~\ref{thm:qc with half malnormal vertices} covers the case
where edge groups are almost malnormal on both sides \begin{com} treated somewhere already?\end{com}
since we can subdivide to put barycenters of edges in $V$.

Another special case where Theorem~\ref{thm:qc with half malnormal vertices} applies
is where $G=G_1*_{{C_1}'=C_2} G_2$ is hyperbolic relative to $\mathbb P$,
and  $C_2 \leq G_2$ is total and relatively quasiconvex in $G$ and almost malnormal in $G_2$.

\begin{proof}
Following the Technical Point in the proof of Theorem~\ref{last}, by splitting certain elements of $\mathbb P$ to obtain $\mathbb P'$ as in Remark~\ref{parab spli}, we can assume that $G$ is hyperbolic relative to $\mathbb P'$ where each $P'\in \mathbb P'$
is elliptic with respect to the action of $G$ on the Bass-Serre tree $T$.
 Since $\mathbb P$ extends $\mathbb P'$ and each $G_v\cap P^g$ is conjugate to an element of $\mathbb P'$,
 we see that  each $G_v$ is quasiconvex in $G$ relative to $\mathbb P'$
by  Theorem~\ref{Yang},
and moreover, since elements of $\mathbb P'$ are vertex groups of elements of $\mathbb P$,
 each $G_v$ is total relative to $\mathbb P'$.\begin{com} The $\mathbb P'$ are the vertex group parts of the $\mathbb P$...\end{com}
Therefore each $G_v$ is hyperbolic relative to a collection $\mathbb P'_v$ of conjugates of elements of $\mathbb P'$.

We argue by induction on the number of edges of $\Gamma$. If $\Gamma$ has no edge the result is contained in the hypothesis. Suppose $\Gamma$ has at least one edge $e$.
If $e$ is separating and $\Gamma=\Gamma_1\sqcup e \sqcup \Gamma_2$ where $e$ attaches $v\in \Gamma^0_1$ to $u\in \Gamma^0_2$ then $G=G_1*_{G_e}G_2$ where $G_i=\pi_1(\Gamma_i)$. Each $G_e$ is the intersection of vertex groups and hence quasiconvex relative to $\mathbb P'$. By Lemma~\ref{QC vertex}, the groups $G_1$ and $G_2$ are quasiconvex in $G$ relative to $\mathbb P'$. Thus $G_i$ is hyperbolic relative to $\mathbb P'_i$ by Remark~\ref{fg for lqc}.

Observe that $T$ contains subtrees $T_1$ and $T_2$ that are the Bass-Serre trees of $\Gamma_1$ and $\Gamma_2$,
and $T-G\tilde e = \{gT_1 \cup gT_2 \ : \ g\in G\}$.
The Bass-Serre tree $\bar T$ of $G_1 *_{G_e} G_2$ is the quotient of $T$ obtained by identifying each $gT_i$ to a vertex.

Since $H$ is relatively finitely generated, there is a finite graph of groups $\Gamma_H$ for $H$,
and a map $\Gamma_H\rightarrow \Gamma$. Removing the edges mapping to $e$ from $\Gamma_H$,
we obtain a collection of finitely many graphs of groups - some over $\Gamma_1$ and some over $\Gamma_2$.
Each component of $\Gamma_H$ corresponds to the stabilizer of some $gT_i$ and is denoted by $H_{gT_i}$,
and since that component is a finite graph with relatively quasiconvex vertex stabilizers,
we see that each $H_{gT_i}$ is relatively quasiconvex in $G_i$ relative to $\mathbb P'_i$ by induction on the number of edges of $\Gamma_H$.

We extend the peripheral structure $\mathbb P'_1$ of $G_1$ to $\mathbb E_1=\{G_1\}$.
Note that now each $H_{gT_1}$ is quasiconvex in $G_1$ relative to $\mathbb E_1$ by Theorem~\ref{Yang}.
 Let $$\mathbb E=\mathbb E_1\cup \mathbb P'_2-\{P\in \mathbb P'_2 \colon P^g\leq {G_e},~\text{for some}~ g\in G_2\}.$$ Observe that $\mathbb E$ extends $\mathbb P'$. Since $G_v$ is total and quasiconvex in $G$ relative to $\mathbb P'$ and $\mathbb E$ extends $\mathbb P'$, the group $G_1$ is total and quasiconvex in $G$ relative to $\mathbb E$ by Theorem~\ref{Yang}. Therefore $G$ is hyperbolic relative to $\mathbb E$ by Theorem~\ref{hyp extend}.

Since $G_1$ is maximal parabolic in $G$, by Theorem~\ref{thm:qc with parabolic vertices} $H$ is quasiconvex in $G$ relative to $\mathbb E$. The graph $\Gamma_H$ shows that $H$
is generated by finitely many hyperbolic elements and  vertex stabilizers $H_{g\bar T_i}$
and each $H_{g\bar T_i} = H_{gT_i}$ which we explained above is relatively quasiconvex in $G_i$.

We now show that $H$ is quasiconvex relative to $\mathbb P'$ and therefore relative to $\mathbb P$ by Theorem~\ref{Yang}. Since $\mathbb E$ extends $\mathbb P'$, by Theorem~\ref{Yang}, it suffices to show that $H\cap E^g$ is quasiconvex relative to $\mathbb P'$ for all $E\in \mathbb E$ and $g\in G$. There are two cases:

Case 1: $E\in \mathbb P'_2$. Now $H\cap E^g$ is a parabolic subgroup of $G$ relative to $\mathbb P'$ and is thus quasiconvex relative to $\mathbb P'$.

Case 2: $E=G_1$. Then $H\cap E^g$ is quasiconvex relative to $\mathbb P'_1$ since $(H\cap E^g)= H_{gT_1}$ is quasiconvex in $G_1^g$ relative to $\mathbb E_1^g=\{G_1^g\}$. Since $E^g=G_1^g$ is quasiconvex relative to $\mathbb P'$, Lemma~\ref{associative} implies that $H\cap E^g$ is quasiconvex relative to $\mathbb P'$.

Now assume that $e$ is nonseparating. Let $u\in U$ and $v\in V$ be the endpoints of $e$.
 Then $G=G_1*_{C^t=D}$ where $G_1$ is the graph of groups over $\Gamma-e$, and $C$ and $D$ are the images of $G_e$ in $G_v$ and $G_u$ respectively. We first reduce the peripheral structure of $G$ from $\mathbb P$ to $\mathbb P'$,
 and we then extend from $\mathbb P'$ to $\mathbb E$ with:
 $$\mathbb E= \{G_v\} \cup \mathbb P'- \{P\in \mathbb P' \colon P^g\leq {G_v},~\text{for some}~ g\in G\}.$$
  $G$ is hyperbolic relative to $\mathbb E$  by Theorem~\ref{hyp extend} as $G_v$ is almost malnormal, total, and quasiconvex
  relative to $\mathbb P$ . The argument follows by induction and Theorem~\ref{thm:qc with parabolic vertices}
  as in the separating case.
  \end{proof}

Theorem~\ref{thm:strongest result} suggests the following
 criterion for relative quasiconvexity:
\begin{conj}\label{conj:generalizing main result}
Let $G$ be hyperbolic relative to $\mathbb P$. Suppose $G$ splits as a finite graph of groups with f.g. relatively quasiconvex edge groups.
 Suppose $H \leq G$ is tamely generated such that each $H_v$ is f.g. for each $v$ in the Bass-Serre tree. Then $H$ is relatively quasiconvex in $G$.
\end{conj}
When the edge groups are separable in $G$, there is
a finite index subgroup $G'$ whose splitting has relatively malnormal edge groups
(see e.g. \cite{HruskaWisePacking,HaglundWiseAmalgams}).
Consequently, if moreover, the edge groups of $G$ are total,
then the induced splitting of $G'$ satisfies the criterion of Theorem~\ref{thm:strongest result},
and we see that Conjecture~\ref{conj:generalizing main result} holds in this case.
In particular, Conjecture~\ref{conj:generalizing main result}  holds when $G$ is virtually special and hyperbolic relative to virtually abelian subgroups, provided that edge groups are also total. We suspect the totalness assumption can be dropped totally.

As a closing thought, consider a hyperbolic $3$-manifold $M$  virtually having a malnormal quasiconvex hierarchy (conjecturally all closed $M$). Theorem~\ref{thm:strongest result} suggests  an alternate approach to the tameness theorem,
which could be reproven by verifying:

\emph{If the intersection of a f.g. $H$ with a malnormal quasiconvex edge group is infinitely generated then $H$ is a virtual fiber.}

{\bf Acknowledgement:} We are extremely grateful to the anonymous f.g. referee whose very helpful corrections and adjustments improved the results and exposition of this paper.

\bibliographystyle{abbrv}

\bibliography{wise.bib}
\end{document}